\documentclass[12pt]{amsart}
\usepackage{latexsym}
\usepackage{amssymb}
\usepackage{mathrsfs}
\usepackage{fullpage}
\usepackage[colorlinks]{hyperref}

\newtheorem{theorem}{Theorem}[section]
\newtheorem{lemma}[theorem]{Lemma}

\newtheorem{corollary}[theorem]{Corollary}

\theoremstyle{definition}

\theoremstyle{remark}
\newtheorem{remark}[theorem]{Remark}

\numberwithin{equation}{theorem}

\newcommand{\GL}{{\mathrm {GL}}}
\newcommand{\PGL}{{\mathrm {PGL}}}

\newcommand{\PSL}{{\mathrm {PSL}}}

\newcommand{\Sp}{{\mathrm {Sp}}}

\newcommand{\Aut}{{\mathrm {Aut}}}
\newcommand{\Out}{{\mathrm {Out}}}

\newcommand{\Irr}{{\mathrm {Irr}}}

\renewcommand{\Im}{{\mathrm {Im}}}

\newcommand{\St}{{\mathrm {St}}}
\newcommand{\Stab}{{\mathrm {Stab}}}

\newcommand{\FF}{{\mathbb F}}

\newcommand{\ta}{\hspace{0.5mm}^{2}\hspace*{-0.2mm}}

\newcommand{\GF}{\mbox{GF}}

\newcommand{\bI}{{\mathbf{I}}}

\newcommand{\bC}{{\mathbf{C}}}
\newcommand{\bF}{{\mathbf{F}}}
\newcommand{\bO}{{\mathbf{O}}}
\newcommand{\bN}{{\mathbf{N}}}
\newcommand{\bZ}{{\mathbf{Z}}}
\newcommand{\Al}{\textup{\textsf{A}}}
\newcommand{\Sy}{\textup{\textsf{S}}}

\begin{document}

\title[Abelian subgroups and the largest character degree]
{Abelian subgroups, nilpotent subgroups, and the largest character
degree of a finite group}

\author{Nguyen Ngoc Hung}
\address{Department of Mathematics, The University of Akron, Akron, OH 44325,
USA} \email{hungnguyen@uakron.edu}

\author{Yong Yang}
\address{Department of Mathematics, Texas State University, San Marcos, TX 78666,
USA} \email{yang@txstate.edu}



\date{\today}

\maketitle

\begin{abstract} Let $H$ be an abelian subgroup of a finite group $G$ and
$\pi$ the set of prime divisors of $|H|$. We prove that $|H
\bO_{\pi}(G)/ \bO_{\pi}(G)|$ is bounded above by the largest
character degree of $G$. A similar result is obtained when $H$ is
nilpotent.
\end{abstract}

\section{Introduction}

Gluck's conjecture \cite{Gluck} asserts that $|G:\bF(G)|\leq b(G)^2$
for every finite solvable group $G$, where $b(G)$ denotes the
largest irreducible character degree of $G$ and $\bF(G)$ denotes the
Fitting subgroup of $G$. Although still open, it has been confirmed
for various cases
\cite{Dolfi-Jabara,Yang,Cossey-Halasi-Maroti-Nguyen} and
furthermore, it was proved by Moreto and Wolf \cite{Moreto-Wolf}
that $|G:\bF(G)|\leq b(G)^3$ for every solvable group $G$. In
\cite{Cossey-Halasi-Maroti-Nguyen}, Cossey et al. provided
considerable evidence showing that the inequality $|G:\bF(G)|\leq
b(G)^3$ might be true for \emph{every} finite group $G$. In
particular, we have $|G:\bF(G)|\leq b(G)^4$ for every $G$, which
then implies that $G$ contains an abelian subgroup of index at most
$b(G)^{8}$, see \cite[Theorems 4 and
8]{Cossey-Halasi-Maroti-Nguyen}.

Can we bound any portion of an abelian subgroup in terms of the
largest character degree $b(G)$?

Our first result is the following. Here we use $\pi(n)$ to denote
the set of prime divisors of a positive integer $n$ and
$\bO_{\pi}(G)$ to denote the largest normal subgroup of $G$ whose
order is divisible by only primes in $\pi$.

\begin{theorem}\label{thm2}
Let $H$ be an abelian subgroup of a finite group $G$ and let
$\pi:=\pi(|H|)$. Then $|H \bO_{\pi}(G)/ \bO_{\pi}(G)| \leq b(G)$.
\end{theorem}

This result does not hold if ``abelian'' is replaced by
``nilpotent''. For example, consider a group $G$ of order $72$ which
is the semidirect product of the dihedral group of order $8$ acting
nontrivially on the elementary abelian group of order $9$. The
largest character degree of $G$ is $4$ while $G$ has a nilpotent
subgroup of order $8$ and $\bO_2(G)=1$. This example also shows that
the bound we obtain in Theorem~\ref{thm1} is tight.

Let $P$ be a Sylow $p$-subgroup of a nonabelian group $G$. It was
proved by Qian and Shi \cite[Theorem 1.1]{QIANSHI} that
$|P/\bO_p(G)|<b(G)^2$. Lewis ~\cite{MLewis} showed that for
$p$-solvable groups $G$, one can strengthen the last result to
$|P/\bO_p(G)| \leq (b(G)^p/p)^{\frac {1} {p-1}}$ and then asked
whether the same statement holds for arbitrary groups. This was
answered affirmatively by Qian and the second author in
\cite{QianYang1}.

In this paper, we strengthen all the previous results from a Sylow
$p$-subgroup to a nilpotent subgroup.

\begin{theorem} \label{thm1}
Let $H$ be a nilpotent subgroup of a nonabelian finite group $G$.
Let $\pi:=\pi(|H|)$ and $p$ the smallest prime in $\pi$. Then $|H
\bO_{\pi}(G)/ \bO_{\pi}(G)| \leq (b(G)^p/p)^{\frac {1} {p-1}}$.
\end{theorem}

The paper is organized as follows. In the next section we bound the
size of nilpotent subgroups of almost simple groups $\Aut(S)$ in
terms of the largest character degree $b(S)$. This is used in
Section~\ref{section-nilpotent} to prove Theorem~\ref{thm1}. Proof
of Theorem~\ref{thm2} is given in Section~\ref{section-abelian}.


\section{Nilpotent subgroups of almost simple groups}

In this section we essentially prove Theorem~\ref{thm1} for almost
simple groups. We start with a situation where the bound is slightly
better.

\begin{lemma}\label{lemma-nilsub-crosscharacteristic}
Let $S$ be a simple group of Lie type defined over a field of
$q=\ell^f$ elements with $\ell$ a prime number. Assume that $S\ncong
\PSL_2(q)$. Let $H$ be a nilpotent subgroup of $\Aut(S)$ and assume
that $H\cap (S\cdot D)$ is not an $\ell$-subgroup of $S\cdot D$
where $D$ is the group of diagonal automorphisms of $S$. Then
$|H|\leq b(S)$.
\end{lemma}

\begin{proof}
We assume that $S\ncong \Sp_4(2)'$, $\ta F_4(2)'$ as these cases can
be checked directly using \cite{Atlas}. We then can find a simple
algebraic group $\mathcal{G}$ of adjoint type and a Frobenius
morphism $F:\mathcal{G}\rightarrow \mathcal{G}$ such that
$G:=\mathcal{G}^F=S\cdot D$. The automomorphism group $\Aut(S)$ is
now a split extension of $G$ by an abelian group, denoted by $A(S)$,
of field and graph automorphisms.

From the hypothesis that $H$ is a nilpotent subgroup of $\Aut(S)$
and $H\cap G$ is not an $\ell$-subgroup of $G$, we deduce that
$H\cap G$ contains an element of order coprime to $\ell$, which
means that $H\cap G$ contains a nontrivial \emph{semisimple
element}, say $s$. Indeed, we can choose $s$ to be central in $H\cap
G$ since $H\cap G$ is nilpotent. We have
\[H\cap G \subseteq \bC_{G}(s).\]

The structure of the centralizers of semisimple elements in simple
groups of Lie type is essentially known, see
\cite{Carter,Fong-Srinivasan,Nguyen,Tiep-Zalesskii} for classical
groups and \cite{Deriziotis,Deriziotis-Liebeck,Deriziotis2,Lubeck}
for exceptional groups. Roughly speaking, the centralizer of a
semisimple element in a finite group of Lie type is ``close'' to a
direct product of finite groups of Lie (possibly other) type of
lower rank. We then bound the size of the nilpotent subgroup $H\cap
G$ of $\bC_{G}(s)$ and use the obvious inequality $|H|\leq |H\cap
G||A(S)|$ to bound $|H|$.

As the arguments for different types of groups are similar, we will
prove only the case $S=\PSL_n(q)$ with $n\geq 3$ as a demonstration.
Note that in this case $G=\PGL_n(q)$. We assume that $S$ is not one
of the small groups $\PSL_n(2)$ with $3\leq n\leq 6$, $\PSL_4(q)$
with $2\leq q\leq 5$, $\PSL_3(q)$ with $3\leq q\leq 11$,
$\PSL_3(16)$, $\PSL_5(3)$, and $\PSL_6(3)$. Indeed, these small
cases can be argued by similar arguments, so we skip the details.

For simplicity, we also denote a preimage of $s$ in $\GL_n(q)$ by
$s$. If the characteristic polynomial of $s$ is a product
$\prod_{i=1}^tf_i^{a_i}(x)$, where each $f_i$ is
 a distinct monic irreducible polynomial over $\mathbb{F}_q$ of degree
$k_i$, then it is well-known that
$${ \bC}_{\GL_n(q)}(s)\cong \GL_{a_1}(q^{k_1})\times \GL_{a_2}(q^{k_2})\times\cdots\times
\GL_{a_t}(q^{k_t}),$$ where $\sum_{i=1}^t a_ik_i=n$ and the number
of $\GL_a(q^k)$ appearing in the product is at most the number of
monic irreducible polynomials over $\mathbb{F}_q$ of degree $k$.

\medskip

1) First we consider the case $q$ is even. By \cite[Table
3]{Vdovin1}, the maximal size of a nilpotent subgroup of $\GL_{2}(q)$
is $q^{2}-1$ and of $\GL_{a}(q)$ with $a\geq 3$ is
$(q-1)q^{a(a-1)/2}$. Therefore the maximal size of a nilpotent
subgroup of $\GL_{a_i}(q^{k_i})$ is at most
$q^{k_i[\frac{a_i(a_i-1)}{2}+1]}$, which in turn implies that the
maximal size of a nilpotent subgroup of ${ \bC}_{\GL_n(q)}(s)$ is at
most
\[q^{\sum_{i=1}^t k_i[\frac{a_i(a_i-1)}{2}+1]}.\] Using induction on
$t$, one can show that \[\sum_{i=1}^t
k_i[\frac{a_i(a_i-1)}{2}+1]\leq \frac{(n-1)(n-2)}{2}+2\] unless
$t=1$ and $(a_1,k_1)=(n,1)$. Since $s\in \PGL_n(q)$ is nontrivial, a
preimage of $s$ in $\GL_n(q)$ is noncentral and so the case $t=1$
and $(a_1,k_1)=(n,1)$ does not happen. We conclude that the maximal
size of a nilpotent subgroup of ${ \bC}_{\GL_n(q)}(s)$ is at most
$q^{\frac{(n-1)(n-2)}{2}+2}$.

Now let $X$ be the preimage of the nilpotent group $H\cap
\PGL_n(q)\subseteq \bC_{\PGL_n(q)}(s)$ in $\GL_n(q)$ and consider
the map
$$\begin{array}{cccc}T:&X&\rightarrow &\bZ(\GL_n(q))\\
&x&\mapsto&z_x=x^{-1}s^{-1}xs\end{array}.$$  It is easy to see that
$T$ is a homomorphism and $z_x=1$ if and only if $x\in X\cap
\bC_{\GL_n(q)}(s)$. Moreover, as $\det(z_x)=1$ for every $x\in X$,
we have $|\Im(T)|\leq n$. We deduce that \[|X|\leq n|X\cap
\bC_{\GL_n(q)}(s)|\leq nq^{\frac{(n-1)(n-2)}{2}+2}.\] Therefore,
 \[|H\cap \PGL_n(q)|\leq \frac{n}{q-1} \cdot q^{\frac{(n-1)(n-2)}{2}+2},\]
 and hence
 \[|H|\leq 2f|H\cap \PGL_n(q)|\leq \frac{2fn}{q-1} \cdot q^{\frac{(n-1)(n-2)}{2}+2}.\]
 Now we use our assumptions on $n$ and $q$ to have
 \[|H|\leq q^{\frac{n(n-1)}{2}}.\] As the Steinberg character of
 $\PSL_n(q)$ has degree $q^{\frac{n(n-1)}{2}}$, we have $|H|\leq
 b(S)$, as wanted.

 \medskip

2) Next we consider the case $q$ is odd. Then the maximal size of a
nilpotent subgroup of $\GL_{2}(q)$ is at most $2(q^{2}-1)$ and of
$\GL_{a}(q)$ with $a\geq 3$ is still $(q-1)q^{a(a-1)/2}$. Arguing as
in the even case, we have that the maximal size of a nilpotent
subgroup of ${ \bC}_{\GL_n(q)}(s)$ is at most $2^{\lfloor
\frac{n}{2}\rfloor}q^{\frac{(n-1)(n-2)}{2}+2}$. As above, we then
have
 \[|H|\leq \frac{2^{\lfloor \frac{n}{2}\rfloor+1}fn}{q-1} \cdot q^{\frac{(n-1)(n-2)}{2}+2}.\]
 Now we use our assumptions on $n$ and $q$ to have
 $|H|\leq q^{\frac{n(n-1)}{2}}\leq b(S)$.
\end{proof}

\begin{theorem}\label{theorem-nilpotent-simple}
Let $S$ be a finite nonabelian simple groups and $H$ a nilpotent
subgroup of $\Aut(S)$. Let $p$ be the smallest prime divisor of
$|H|$. Then \[p^{\frac 1 {p-1}}|H| \leq b(S)^{\frac p {p-1}}\]
unless $S=\PSL_2(8)$ and $H$ is a Sylow $3$-subgroup of
$\Aut(\PSL_2(8))$.
\end{theorem}

\begin{proof}

1) {\textbf{Alternating groups.}} Let $S=\Al_n$ with $n\geq 7$. We
claim that $|H|\leq b(S)$ for every nilpotent subgroup $H$ of
$\Aut(S)$, and hence the theorem follows.

It is straightforward to check the statement for $7\leq n\leq 12$,
so we assume that $n\geq 13$. In particular $\Aut(\Al_n)=\Sy_n$. By
\cite[Theorem 2.1]{Vdovin1}, a nilpotent subgroup of $\Sy_n$ of
maximal order is conjugate to $Syl_2(\Sy_n)$ if $n\not\equiv 3 \bmod
4$ and to $Syl_2(\Sy_{n-3})\times \langle n-2,n-1,n\rangle$ if
otherwise. Therefore,
\[|H|\leq 2^{[n/2]+[n/2^2]+\cdots}\] if $n\not\equiv 3 \bmod
4$ and
\[|H|\leq 3\cdot 2^{[(n-3)/2]+[(n-3)/2^2]+\cdots}\] if $n\equiv 3 \bmod
4$. By induction on $n$, one can check that both
$2^{[n/2]+[n/2^2]+\cdots}$ and $3\cdot
2^{[(n-3)/2]+[(n-3)/2^2]+\cdots}$ are at most
$|\Al_n|^{1/3}=(n!/2)^{1/3}$ for $n\geq 13$. So $|H|\leq
|\Al_n|^{1/3}$ for $n\geq 13$. Using \cite[Theorem
12]{Cossey-Halasi-Maroti-Nguyen} (see \cite[Theorem
2.1]{Hung-Lewis-Schaeffer} also), we deduce that $|H|<b(\Al_n)$ for
$n\geq 13$ and the claim is proved.

\medskip

2) {\textbf{Classical groups other than $\PSL_2(q)$.}} Since the
treatments for different families of groups are similar, we present
here only the case $S=\PSL_n(q)$ for $n\geq 3$ and $q=\ell^f$.

By Lemma \ref{lemma-nilsub-crosscharacteristic}, we may assume that
$H\cap \PGL_n(q)$ is an $\ell$-subgroup of $\PGL_n(q)$. As the
automorphism group $\Aut(S)$ is well-known (see Theorem 2.5.12 of
\cite{Gorenstein-Lyons-Solomon} for instance), we know that
$\Aut(\PSL_n(q))$ is a split extension of $\PGL_n(q)$ by an abelian
group of order $2f$. Assume that
\[|H|=x|H\cap\PGL_n(q)|\leq xq^{n(n-1)/2}\] where $x\mid 2f$. If $x=1$ then $|H|\leq b(S)$ and we are done.
So we may assume that $x>1$. Since $p$ is the smallest prime divisor
of $|H|$, one has $p\leq \min\{\ell,x\}$. Note that $b(S)\geq
q^{n(n-1)/2}\geq \ell^{3f}$. Therefore, to prove the desired
inequality, it is enough to show
\[\min\{\ell,x\}\cdot x^{\min\{\ell,x\}-1}\leq \ell^{3f}.\]
If $\ell\geq x$ then  $\min\{\ell,x\}\cdot
x^{\min\{\ell,x\}-1}=x\cdot x^{x-1}=x^x\leq\ell^x$. As $x\leq 2f$,
it follows that $\min\{\ell,x\}\cdot x^{\min\{\ell,x\}-1}\leq
\ell^{2f}<\ell^{3f}$, as wanted. So we can assume that $\ell<x$.
Then $\min\{\ell,x\}\cdot x^{\min\{\ell,x\}-1}=\ell\cdot x^{\ell-1}$
and hence we need to show that
\[\ell\cdot x^{\ell-1}\leq \ell^{3f},\] which is equivalent to $x^{\ell-1}\leq
\ell^{3f}$. Since $x\leq 2f$, it is enough to show that
$(2f)^\ell\leq \ell^{3f}$. This inequality is elementary with the
note that $2\leq\ell<2f$.

\medskip

3) {\textbf{Exceptional groups of Lie type.}} The arguments for
exceptional groups are indeed similar to those for classical groups,
so we consider only the case $S=\ta E_6(q)$ as an example.

By Lemma \ref{lemma-nilsub-crosscharacteristic}, we may assume that
$H\cap (\ta E_6)_{ad}(q)$ is an $\ell$-subgroup of order $q^{36}$ of
$(\ta E_6)_{ad}(q)$. The automorphism group $\Aut(S)$ is a split
extension of $(\ta E_6)_{ad}(q)$ by the cylic group of order $f$.
Assume that
\[|H|=x|H\cap(\ta E_6)_{ad}(q)|\leq xq^{36}\] where $x\mid f$. If $x=1$ then $|H|\leq b(S)$ and we are done.
So we may assume that $x>1$. Since $p$ is the smallest prime divisor
of $|H|$, one has $p\leq \min\{\ell,x\}$. Note that $b(S)\geq
q^{36}= \ell^{36f}$. Therefore, to prove the desired inequality, it
is enough to show
\[\min\{\ell,x\}\cdot x^{\min\{\ell,x\}-1}\leq \ell^{36f},\]
but this can be argued as above.

\medskip

4) {\textbf{Sporadic groups.}} Since $H\cap S$ is a nilpotent group,
it contains a central element $z$ of order $p_1p_2\cdots p_k$ where
$p_1, p_2,...,p_k$ are all distinct primes dividing $|H\cap S|$. Now
we have \[|H\cap S|\leq |\mathbf{C}_S(z)|_{p_1,p_2,...,p_k}.\]
Checking the orders of centralizers of these elements in
\cite{Atlas}, we see that $|\mathbf{C}_S(z)|_{p_1,p_2,...,p_k}$ is
maximal when $k=1$, which means that $|H\cap S|$ is maximal when it
is a Sylow subgroup of $S$.

We also can check from \cite{Atlas} that the order of a Sylow
subgroup of $S$ is at most $b(S)/2$. Therefore \[|H\cap S|\leq
b(S)/2.\] Note that $|\Aut(S):S|\leq 2$. So $|H:(H\cap S)|\leq 2$.
We deduce that $|H|\leq 2|H\cap S|\leq b(S)$, which implies the
desired inequality.

\medskip

5) {\textbf{$S=\PSL_2(q)$ for $q=\ell^f\geq 5$.}} As the small cases
$S=\PSL_2(5\leq q\leq 9)$ can be checked using \cite{GAP}, we assume
that $q\geq 11$. Note that $b(S)=q+1$. As the case where $H\cap
\PGL_2(q)$ is a subgroup of an unipotent subgroup of $\PGL_2(q)$ can
be handled as before, we assume that $H\cap \PGL_2(q)$ contains a
nontrivial \emph{central} semisimple element, say $s$. For
simplicity, we also denote a preimage of $s$ in $\GL_2(q)$ by $s$.

First assume that $|H|$ is even or equivalently $p=2$. Note that the
nilpotent subgroup $H\cap \PGL_2(q)$ of $\PGL_2(q)$ has order at
most $2(q+1)$ (see \cite{Vdovin1} for instance). We deduce that
$|H|\leq 2f(q+1)$ and hence
\[p^{\frac 1 {p-1}}|H|=2|H| \leq 4f(q+1).\]
As $q\geq 11$, we have $4f\leq \ell^f+1=q+1$ and therefore
\[p^{\frac 1 {p-1}}|H|\leq (q+1)^2=b(S)^p,\] as wanted.

From now on we assume that $|H|$ is odd. Then $H\cap
\PGL_2(q)\subseteq \bC_{\PGL_2(q)}(s)$ is an odd-order nilpotent
subgroup of $\PGL_2(q)$. Note that
\[\bC_{\GL_2(q)}(s)\cong\GL_1(q^2) \text{ or }\GL_1(q)\times
\GL_1(q)\] and $\bC_{\PGL_2(q)}(s)$ is an extension of
$\bC_{\GL_2(q)}(s)/\bZ(\GL_2(q))$ by a trivial group or a cyclic
group of order 2. The oddness of $|H|$ then implies that $H\cap
\PGL_2(q)$ is a subgroup of a cyclic subgroup (of order $q-\epsilon$
for $\epsilon=\pm1$) of $\PGL_2(q)$. Therefore, we would have
$|H|\leq q+1=b(S)$ if $H\subseteq \PGL_2(q)$. So we assume that
$H\nsubseteq \PGL_2(q)$. This in particular implies that $f\geq 3$.
We also assume furthermore that $q\neq 2^6$ as this case can be
checked easily. Also, recall from the hypothesis that $q\neq 2^3$.

Let $\nu$ be a generator of $\FF^*_{q^2}$, the multiplicative group
of $\FF_{q^2}$. Then $H\cap \PGL_2(q)$ is isomorphic to a subgroup
of $\langle \nu^{q+\epsilon}\rangle$. A field automorphism $\tau$
acting on $H\cap \PGL_2(q)$ by rasing each entry in the matrix to
its $\ell^{f_1}$th power (with $1\leq f_1\leq f$) then acts on
$\langle \nu^{q+\epsilon}\rangle$ by rasing each element to its
$\ell^{f_1}$th power. For that reason, we will identify $H\cap
\PGL_2(q)$ with its isomorphic image in $\langle
\nu^{q+\epsilon}\rangle$.

Our assumption on $q$ and $f$ guarantees that there exists a
primitive prime divisor, say $r$, of $q-\epsilon=\ell^f-\epsilon$.
(When $\epsilon=-1$, one can take $r$ to be a primitive prime
divisor of $\ell^{2f}-1$.) Then $\ell$ has order $f$ or $2f$ modulo
$r$, which implies that $r>f$.

Assume that $\nu^{(q^2-1)/r}\in H\cap \PGL_2(q)$. Recall that
$H\nsubseteq \PGL_2(q)$. We have that $H$ is an extension of $H\cap
\PGL_2(q)$ by some nontrivial field automorphisms. Let $\tau$ be
such an automorphism and suppose that $\tau$ acts on $\langle
\nu^{q+\epsilon}\rangle$ by rasing each element to its
$\ell^{f_1}$th power. Note that the order of this automorphism is
$f/\gcd(f_1,f)\leq f$. Since $H$ is nilpotent and $r>f$, $\tau$ must
act trivially on $\nu^{(q^2-1)/r}$. We then have
\[\nu^{\ell^{f_1}(q^2-1)/r}=\nu^{(q^2-1)/r},\] which implies that $r\mid
(\ell^{f_1}-1)$. This is impossible as $r$ is a primitive prime
divisor of either $\ell^f-1$ or $\ell^{2f}-1$.

So we must have that the element $\nu^{(q^2-1)/r}$ of order $r$ is
not in $H\cap \PGL_2(q)$, then
\[|H|\leq |H\cap \PGL_2(q)|f\leq (q-\epsilon)f/r<q-\epsilon\leq
b(S),\] and we are done.
\end{proof}


\section{Proof of Theorem \ref{thm1}}\label{section-nilpotent}


We first collect some lemmas which are needed in the proof of
Theorem \ref{thm1}. The next lemma is probably known somewhere else.

\begin{lemma}  \label{lem1}
Let $G$ be a nilpotent permutation group of degree $n$ and let $p$
be the smallest prime divisor of $|G|$. Then $|G| \leq ({p^{\frac 1
{p-1}}})^{n-1}$.
\end{lemma}

\begin{proof}
As mentioned earlier, Vdovin proved that the maximal size of a
nilpotent subgroup of $\Sy_n$ is
\[2^{[n/2]+[n/2^2]+\cdots}\] if $n\not\equiv 3 \bmod
4$ and
\[3\cdot 2^{[(n-3)/2]+[(n-3)/2^2]+\cdots}\] if $n\equiv 3 \bmod
4$. Consequently we have $|G|\leq 2^{n-1}$ and so the lemma is
proved when $|G|$ is even. We will follow Vdovin's idea to prove our
lemma.

Let $O_1,O_2,...$ be the orbits of the action of $\bZ(G)$ on the set
$\{1,2,...,n\}$. If a permutation $\sigma\in G$ moves an element $k$
in $O_i$ to $O_j$, it is easy to see that $O_i$ and $O_j$ have the
same cardinality. Therefore $G/\bZ(G)$ permutes the orbits of the
same cardinality in $\{O_1,O_2,...\}$. Let $T_k$ be the union of
orbits of cardinality $k$ and let $n_i:=|T_k|$. It follows that
\[n=n_1+n_2+\cdots\] and
\[G\leq \Sy_{n_1}\times \Sy_{n_2}\times \cdots.\]

Suppose that there are at least two orbits of different
cardinalities. Then the conclusion of the previous paragraph shows
that $G$ is a subgroup of $\Sy_{m_1}\times \Sy_{m_2}$ for
$m_1+m_2=n$ and $m_1,m_2<n$. Using induction, we have
\[|G|\leq (p^{\frac{1}{p-1}})^{m_1-1}(p^{\frac{1}{p-1}})^{m_2-1},\]
which is clearly smaller than $(p^{\frac{1}{p-1}})^{n-1}$.

So we can now assume that all of the orbits have the same size, say
$s$. That means there are exactly $n/s$ orbits:
$O_1,O_2,...,O_{n/s}$. As mentioned above that $G/\bZ(G)$ permutes
these $O_1,O_2,...,O_{n/s}$, if $G/\bZ(G)$ has more than one orbits
on $\{O_1,O_2,...,O_{n/s}\}$, we can again see that $G$ can be
considered as a subgroup of $\Sy_{m_1}\times \Sy_{m_2}$ for
$m_1+m_2=n$ and $m_1,m_2<n$, and so the lemma follows. So we assume
that $G/\bZ(G)$ acts transitively on $\{O_1,O_2,...,O_{n/s}\}$.

We claim that $|\bZ(G)|=s$, which is equivalent to
$\Stab_{\bZ(G)}(1)$ is trivial. Without loss we assume that $1\in
O_1$ and let $\tau\in \Stab_{\bZ(G)}(1)$. Let $k$ be an arbitrary
element in $\{1,2,...,n\}$ and assume that $k\in O_i$. As $G/\bZ(G)$
acts transitively on $\{O_1,O_2,...,O_{n/s}\}$ and $\bZ(G)$ acts
transitively on $O_1$, there exist $\tau_1\in G$ and $\tau_2\in
\bZ(G)$ such that $(k^{\tau_1})^{\tau_2}=1$. Then
$k=(1^{\tau_2^{-1}})^{\tau_1^{-1}}$ and it follows that
\[k^\tau=((1^{\tau_2^{-1}})^{\tau_1^{-1}})^\tau=
((1^\tau)^{\tau_2^{-1}})^{\tau_1^{-1}}=(1^{\tau_2^{-1}})^{\tau_1^{-1}}=k.\]
Here we note that $\tau\in\bZ(G)$ commutes with $\tau_1$ and
$\tau_2$. We have shown that $\tau$ fixes every element in
$\{1,2,...,n\}$, which means that $\tau$ is trivial, as claimed.

Note that, as $G$ is nilpotent, we have $s>1$ and so using
induction, we deduce that
\[|G|=|\bZ(G)||G/\bZ(G)|\leq s(p^{\frac{1}{p-1}})^{\frac{n}{s}-1}.\]

To prove the lemma, we now need to show that
\[s(p^{\frac{1}{p-1}})^{\frac{n}{s}-1}\leq
(p^{\frac{1}{p-1}})^{n-1},\] which is equivalent to
\[s\leq (p^{\frac{1}{p-1}})^{n-\frac{n}{s}}.\]
Note that $p\leq n$ and hence $p^{\frac{1}{p-1}}\geq
n^{\frac{1}{n-1}}$. Thus it is enough to show
\[s\leq (n^{\frac{1}{n-1}})^{n-\frac{n}{s}},\] which is equivalent
to
\[\frac{n}{n-1} (1-\frac{1}{s})\ln n\geq \ln s.\]
This last inequality follows from the fact that the function
$f(x)=\dfrac{x\ln x}{x-1}$ is increasing on $[2,\infty)$ and $2\leq
s\leq n$.
\end{proof}

\begin{lemma}  \label{lem2}
Let $S$ be a transitive solvable permutation group on $\Omega$ with
$|\Omega|=n$. If $|S|$ is odd, then S has a regular orbit on the
power set $\mathscr{P}(\Omega)$ of $\Omega$.
\end{lemma}
\begin{proof}
This is Gluck's Theorem ~\cite[Corollary 5.7]{manz/wolf}.
\end{proof}

\begin{lemma}\label{coprimeaction}
Let $N$ be a nontrivial nilpotent $\pi$-group, where $\pi$ is a set
of primes, and assume that $N$ acts faithfully on a $\pi'$-group
$H$. Then there exists $x \in H$ such that $|\bC_N(x)| \leq
(|N|/p)^{1/p}$, where $p$ is the smallest member of $\pi$.
\end{lemma}

\begin{proof}
This is due to Isaacs ~\cite[Theorem B]{IMI2}.
\end{proof}

We now prove the second main result, which we restate for reader's
convenience. We will frequently use the following fact without
mention: if $L$ is a subgroup or a quotient group of $G$, then $b(L)
\leq b(G)$. Also, the Frattini subgroup of $G$ is the intersection
of all maximal subgroups of $G$ and is denoted by $\Phi(G)$. We use
$|G|_{nil}$ to denote the maximum order of all the nilpotent
subgroups of $G$ and $|G|_{abelian}$ to denote the maximum order of
all the abelian subgroups of $G$.

\begin{theorem} \label{thm1-repeat}
Let $H$ be a nilpotent subgroup of a nonabelian finite group $G$.
Let $\pi:=\pi(|H|)$ and $p$ the smallest prime in $\pi$. Then $|H
\bO_{\pi}(G)/ \bO_{\pi}(G)| \leq (b(G)^p/p)^{\frac {1} {p-1}}$.
\end{theorem}

\begin{proof}


Suppose that $\bO_{\pi}(G) >1$. If $G/\bO_{\pi}(G)$ is abelian, then
$H \leq \bO_{\pi}(G)$ and $|H\bO_{\pi}(G)/\bO_{\pi}(G)|=1$, and we
are done. If $G/\bO_{\pi}(G)$ is nonabelian, then induction yields
the required result.

Thus we may assume from now on that $\bO_{\pi}(G)=1$.

Suppose that $\Phi(G)>1$ and let $N$ be a minimal $G$-invariant
subgroup of $\Phi(G)$. Then $N$ is a $r$-group for some $r\not\in
\pi$.  Assume that $T/N=\bO_{\pi}(G/N)>1$. Since $(|T/N|,|N|)=1$,
there exists a Hall $\pi$-subgroup $K$ of $T$ such that $T=K \ltimes
N$.  Since all Hall $\pi$-subgroups of $T$ are conjugate, we deduce
by the Frattini argument that $G = T \bN_G(K)= N \bN_G(K)$. As
$N\leq \Phi(G)$, it follows that $G=\bN_G(K)$. Consequently
$\bO_{\pi}(G)\geq K>1$, a contradiction. Thus we must have
$\bO_{\pi}(G/N)=1$.

Assume that $G/N$ is abelian. Then $G$ is nilpotent and so it
possesses a normal abelian Hall $\pi$-subgroup. This implies that
$H\leq \bO_{\pi}(G)=1$, which makes the theorem obvious. Assume that
$G/N$ is nonabelian. Then one can use induction with the note that
$b(G)\geq b(G/N)$ and $|(HN/N \bO_{\pi}(G/N))/\bO_{\pi}(G/N)|=|H|$.
Hence we may assume that $\Phi(G)=1$.

Assume  that all minimal normal subgroups of $G$ are solvable. Let
$\bF$ be the Fitting subgroup of $G$. Since
$\Phi(G)=\bO_{\pi}(G)=1$, we see that $G = \bF \rtimes A$ is a
semidirect product of an abelian $\pi'$-group $\bF$ and a group $A$
which is isomorphic to $G/\bF$. Clearly, $\bC_G(\bF) = \bC_A(\bF)
\times \bF$ and $\bC_A(\bF) \vartriangleleft G$. Since $\bF$
contains all the minimal normal subgroups of $G$, we conclude that
$\bC_A(\bF) = 1$, and hence, $\bC_G(\bF) = \bF$.

Let us investigate the subgroup $K = H\bF$. Since $\bO_{\pi}(K)$
centralizes the $\pi'$-group $F$ and hence $\bO_{\pi}(K) \leq
\bC_G(\bF)=\bF$, it follows that $\bO_{\pi}(K)=1$. Assume that $K<
G$. Note that $H>1$ and $H\bF$ is nonabelian, the result follows by
induction. Therefore, we may assume that $G= K = H\bF$. Observe that
$G = H\bF$ is solvable and $H$ acts faithfully on the abelian
$p'$-group $\bF$. By Lemma~\ref{coprimeaction}, there exists some
linear $\lambda \in \Irr(\bF)$ such that $|\bC_H(\lambda)| \leq
(|H|/p)^{\frac 1 p}$, and so that $b(G)\geq
p^{\frac{1}{p}}|H|^{\frac {p-1} {p}}$. Therefore, the theorem holds
in this case.

Now we assume that $G$ has a nonsolvable minimal normal subgroup $V$.
Set $V = V_1 \times \cdots \times V_k$,
where $V_1, \ldots, V_k$ are isomorphic nonabelian simple groups.
Let us investigate the subgroup $K = H(V \times \bC_G(V))$.

Since $V$ is a direct product of nonabelian simple groups,
$\bO_{\pi}(V)=1$. This implies that $V \cap \bO_{\pi}(K)=1$. Since
$V$ and $\bO_{\pi}(K)$ are both normal in $K$, this implies that
$\bO_{\pi}(K)$ centralizes $V$, and so, $\bO_{\pi}(K) \leq
\bC_G(V)$, and hence, $\bO_{\pi} (K) \le \bO_{\pi}(\bC_G(V))$. Since
$\bC_G(V)$ is normal in $G$, we see that $\bO_{\pi} (\bC_G(V)) \le
\bO_{\pi} (G) = 1$. Thus, we conclude that $\bO_{\pi}(K) = 1$.
Therefore we may assume by induction that $K=G$, i.e., $G/(V \times
\bC_G(V ))$ is a nilpotent $\pi$-group.

Clearly $\bO_{\pi}(\bC_G(V))=1$. If $\bC_G(V)$ is not ableian, then
by induction there exists $\psi \in \Irr(\bC_G(V))$ such that
$\psi(1)\geq p^{\frac{1}{p}} |H \cap \bC_G(V)|^{\frac{p-1}{p}}$. If
$\bC_G(V)$ is ableian, then all $\psi \in \Irr(\bC_G(V))$ has degree
$1$ and $|H \cap \bC_G(V)|=1$. Thus in all cases, we have
$\psi(1)\geq |H \cap \bC_G(V)|^{\frac{p-1}{p}}$.

Let $\chi_i \in \Irr(V_i)$ such that $\chi_i(1) = b(V_i)$ and set
$\chi = \chi_1 \times \cdots \times \chi_k$. Clearly $\chi \in
\Irr(V)$ and $\chi(1) = \chi_1^k(1)$. Note that $G/(V \times
\bC_G(V)) \leq \Out(V)$ and $\Out(V) \cong \Out(V_1) \wr S_k$, we
have $G /\bC_G(V) \lesssim \Aut(V_1) \wr S_k$. By Lemma~\ref{lem1},
we have \[|H\bC_G(V)/\bC_G(V)| \leq p^{\frac{k-1}{p-1}}
(|\Aut(V_1)|_{nil})^k,\] where $|X|_{nil}$ denotes the largest size
of a nilpotent subgroup of $X$.

Suppose that $V_1 \not \cong A_1(2^3)$ or $p \neq 3$. By
Theorem~\ref{theorem-nilpotent-simple}, we have
\begin{align*}|H\bC_G(V)/\bC_G(V)| &\leq (p^{\frac 1 {p-1}})^{k-1}
(|\Aut(V_1)|_{nil})^k \\
&\leq (p^{\frac 1 {p-1}})^{-1} (\chi_1(1)^{\frac p {p-1}})^{k}\\
&=p^{\frac{-1}{p-1}} \chi(1)^{\frac{p}{p-1}}.\end{align*} Note that
$\chi\times \psi$ is an irreducible character of $V\times \bC_G(V)$.
We get that
\begin{align*} b(G)&\geq b(V\times \bC_G(V)) \geq
\psi(1)\chi(1)\\
&\geq |H\cap \bC_G(V)|^{\frac{p-1}{p}}
(|H\bC_G(V)/\bC_G(V)|)^{\frac{p-1}{p}}p^{\frac{1}{p}}\\
&=(|H|)^{\frac{p-1}{p}}p^{\frac{1}{p}},\end{align*} and we are done.

It remains to consider the case $V_i \cong A_1(2^3)$ and $p=3$. This
implies that $H$ is of odd order. By Atlas~\cite{Atlas}, we may take
$\mu_i, \nu_i \in \Irr(V_i)$ such that $\mu_i(1)=7$, $\nu_i(1)=9$
and $\bI_{\bN_G(V_i)}(\mu_i) = \bI_{\bN_G (V_i)}(\nu_i)= \bC_G(V_i)
\times V_i$. Using Lemma~\ref{lem2} and the fact that $H$ is of odd
order, we see that there exists $\chi = \prod_i(\chi_i) \in \Irr(V)$
such that  $\chi_i \in \{\mu_i, \nu_i\}$ and $\bI_G(\chi) = V\times
\bC_G(V)$. Clearly  $\chi(1) \geq 7^k$.

Since $\bI_G(\chi) = V\times \bC_G(V)$, $\chi \times \psi \in
\Irr(V\times \bC_G(V))$ induces an irreducible character of $G$, and
this implies that
\[|G : (V \times \bC_G(V))| \chi(1) \psi(1) \leq  b(G).\]
Since $k = |G : \bN_G(V_1)|$, we may write $|G : (V \times
\bC_G(V))| = ka$.



We have $|H| \leq ka9^k |H \cap \bC_G (V)|$. Since $9^k \le
(7^k)^{3/2}/3^{1/2}$ and $ka \le (ka)^{3/2}$ as both $k$ and $a$ are
both at least 1, it follows that

\[|H| \le (ka)^{3/2} ((7^k)^{3/2}/ 3^{1/2}) |H \cap \bC_G (V)|.\]
Now since $7^k \le \chi(1)$, $|G:V \times \bC_G(V)| = ka$ and $|H
\cap \bC_G (V)| \le \psi (1)^{3/2}$, the previous inequality yields
\[|H| \le |G:V \times \bC_G(V)|^{3/2} \chi(1)^{3/2}
\psi(1)^{3/2}/3^{1/2}.\] Finally, we know that $|G:V \times \bC_G
(V)| \chi (1) \psi (1) \le b(G)$, and so we have \[3^{1/2} \cdot |H|
\le (b(G))^{3/2},\] which is the desired inequality.
\end{proof}

The following consequence of Theorem~\ref{thm1} is the main result
of \cite{QianYang1}.

\begin{corollary}\label{onep}
Let $P$ be a Sylow $p$-subgroup of a non-abelian finite group $G$.
Then $|P/ \bO_{p}(G)| \leq (b(G)^p/p)^{\frac {1} {p-1}}$.
\end{corollary}

\begin{proof}
This is special case of Theorem~\ref{thm1}.
\end{proof}


\section{Abelian subgroups and proof of Theorem
\ref{thm2}}\label{section-abelian}
We start this section with a variant of Lemma~\ref{lem2} for abelian
groups.

\begin{lemma}  \label{lem3}
Let $S$ be a transitive solvable permutation group on $\Omega$ with
$|\Omega|=n$. If $S$ is abelian, then $S$ has a regular orbit on the
power set $\mathscr{P}(\Omega)$ of $\Omega$.
\end{lemma}

\begin{proof}
Note that if for $A, B \in \mathscr{P}(\Omega)$ we define $A+B=(A
\cup B)-(A \cap B)$, then $\mathscr{P}(\Omega)$ with this addiction
becomes a $\GF(2)$-module. By \cite[Corollary 5.7]{manz/wolf} we
know that the result holds if $S$ is primitive and the induction
follows by \cite[Theorem 2.10]{Carlip}.
\end{proof}

\begin{lemma}  \label{lem4}
Let $G=H \wr S$, where $H$ is nontrivial and $S$ is a permutation
group of degree $n$. Let $A$ be an abelian subgroup of $G$. Assume
that the maximum order of an abelian subgroup of $H$ is $a$, then
$|A| \leq a^n$.
\end{lemma}

\begin{proof} Suppose that $A=B\times T$, where $B$ is an abelian
subgroup of ${H\times\cdots\times H}$ ($n$ times) and $T$ is an
abelian subgroup of $\Sy_n$. As the lemma is obvious when $T$ is
trivial, we assume that $|T|>1$. Let $m$ be the integer such that
\[2^m\leq |T|<2^{m+1}.\]

The maximal size of an abelian subgroup of $\Sy_m$ is at most
$3^{m/3}$, see \cite[Theorem 1.1]{Vdovin2}. As $3^{m/3}<2^m$ and $T$
is abelian, we deduce that
\[\Stab_{\{1,2,...,n\}}(T)<n-m.\]
It follows that $B$ can be considered as a subgroup of
${H\times\cdots\times H}$ ($n-m-1$ times), which implies that
\[|B|\leq a^{n-m-1}.\] We now have
\[|A|=|B|\cdot|T|< a^{n-m-1}\cdot 2^{m+1}\leq a^n,\] as desired.
\end{proof}

\begin{theorem}\label{abeliancase}
Let $A$ be an abelian subgroup of $\Aut(S)$, where $S$ is a
nonabelian simple group. Then $|A|\leq b(S)$ unless $S=\Al_5$.
\end{theorem}

\begin{proof}
It was already shown in the proof of
Theorem~\ref{theorem-nilpotent-simple} that the order of a nilpotent
subgroup of $\Aut(\Al_n)$ with $n\geq 7$ is at most $b(\Al_n)$. This
also holds for sporadic simple groups. Therefore from now on we
assume that $S$ is a simple group of Lie type defined over a field
of $q=\ell^f$ elements.

First suppose that $S=E_6(q)$ or $\ta E_6(q)$. Then $\Out(S)\leq
6f$. By \cite[Theorem A]{Vdovin2}, it follows that \[|A|<
|S|^{1/3}|\Out(S)|\leq 6f|S|^{1/3}|<6fq^{26}.\] The theorem then
follows as $b(S)\geq \St_S(1)= q^{36}$, where $\St_S$ denotes the
Steinberg character of $S$. For other exceptional groups, the
arguments are similar with the note that $|\Out(S)|\leq 2f$.

Since the arguments for different families of classical groups are
similar, we provide the proof for only the linear groups.

First suppose that $S=\PSL_3(q)$. The maximal order of an abelian
subgroup of $S$ is $q^2+q+1$ if $3\nmid q-1$ and $q^2$ if $3\mid
(q-1)$. Since $|\Out(S)|=2f(3,q-1)$, it follows that $|A|<6fq^2$.
The theorem then follows unless $q=2,3,5,9$. For these exceptional
cases, one can check the inequality directly using \cite{GAP}.

Now suppose that $S=\PSL_n(q)$ for $n\geq 4$. The maximal order of
an abelian subgroup of $S$ is $q^{\lfloor n^2/4\rfloor}$. Therefore
\[|A|\leq q^{\lfloor n^2/4\rfloor} 2f \gcd(n,q-1).\]
It is now easy to check that $q^{\lfloor n^2/4\rfloor} 2f
\gcd(n,q-1)\leq q^{n(n-1)/2}$, and the theorem is good in this case.

Lastly we suppose that $S=\PSL_2(q)$ with $q>5$. Recall that
$\PGL_2(q)=\PSL_2(q)$ if $q$ is even and $\PGL_2(q)=\PSL_2(q)\cdot
2$ if $q$ is odd. Moreover $\Out(\PGL_2(q))$ is a cyclic group of
order $f$ of field automorphisms of $S$. Assume that $A\cong B\times
C$ where $B$ is an abelian subgroup of $\PGL_2(q)$ and $C$ is a
cyclic subgroup of $\Out(\PGL_2(q))$ of order $f_1$, which is a
divisor of $f$. Then every matrix in $B$ is fixed by the field
automorphism $x\mapsto x^{\ell^{f/f_1}}$, which implies that $B\leq
\PGL_2(\ell^{f/f_1})$ and hence $|B|\leq \ell^{f/f_1}+1$. Therefore
\[|A|\leq (\ell^{f/f_1}+1)f_1\leq q+1=b(S),\] as desired.
\end{proof}

We are now ready to prove Theorem~\ref{thm2}. It is not surprise
that the proof follows the same ideas as in the proof of
Theorem~\ref{thm1}.

\begin{theorem}\label{abeliancase}
Let $H$ be an abelian subgroup of a finite group $G$ and let
$\pi:=\pi(|H|)$. Then $|H \bO_{\pi}(G)/ \bO_{\pi}(G)| \leq b(G)$.
\end{theorem}

\begin{proof}
The theorem is obvious when $H=1$. So we assume that $H$ is
nontrivial. Using induction, we may assume that
$\bO_{\pi}(G)=1$. Also, by using the same arguments as in the proof of Theorem ~\ref{thm1-repeat}, we may assume that $\Phi(G)=1$.

%



We first assume that all minimal normal subgroups of $G$ are solvable. As
before, let $\bF$ be the Fitting subgroup of $G$. Since
$\Phi(G)=\bO_{\pi}(G)=1$, $G = \bF \rtimes A$ is a semidirect
product of an abelian $\pi'$-group $\bF$ and a group $A$ which is
isomorphic to $G/\bF$. Clearly, $\bC_G(\bF) = \bC_A(\bF) \times \bF$
and $\bC_A(\bF) \vartriangleleft G$. Since $\bF$ contains all the
minimal normal subgroups of $G$, we conclude that $\bC_A(\bF) = 1$,
and hence, $\bC_G(\bF) = \bF$.

Let us investigate the subgroup $K = H\bF$. Since $\bO_{\pi}(K)$
centralizes the $\pi'$-group $\bF$ and hence $\bO_{\pi}(K) \leq
\bC_G(\bF)=\bF$, it follows that $\bO_{\pi}(K)=1$. Assume that $K<
G$. Note that $H>1$ by our assumption  and $H\bF$ is nonabelian, the
result follows by induction. Therefore, we may assume that $G= K =
H\bF$. Observe that $G = H\bF$ is solvable and $H$ acts faithfully
on the abelian $p'$-group $\bF$. Since $H$ is abelian, there exists some linear $\lambda \in \Irr(\bF)$
such that $|\bC_H(\lambda)| =1$, and so that $b(G)\geq |H|$.
Therefore, the theorem holds in this case.

We now assume that $G$ has a nonsolvable minimal normal subgroup $V$.
Set $V = V_1 \times \cdots \times V_k$,
where $V_1, \ldots, V_k$ are isomorphic nonabelian simple groups.
Let us investigate the subgroup $K = H(V \times \bC_G(V))$.

Since $V$ is a direct product of nonabelian simple groups,
$\bO_{\pi}(V)=1$. This implies that $V \cap \bO_{\pi}(K)=1$. Since
$V$ and $\bO_{\pi}(K)$ are both normal in $K$, this implies that
$\bO_{\pi}(K)$ centralizes $V$, and so, $\bO_{\pi}(K) \leq
\bC_G(V)$, and hence, $\bO_{\pi} (K) \le \bO_{\pi}(\bC_G(V))$. Since
$\bC_G(V)$ is normal in $G$, we see that $\bO_{\pi} (\bC_G(V)) \le
\bO_{\pi} (G) = 1$. Thus, we conclude that $\bO_{\pi}(K) = 1$.
Therefore we may assume by induction that $K=G$, i.e., $G/(V \times
\bC_G(V ))$ is an abelian $\pi$-group.

Clearly $\bO_{\pi}(\bC_G(V))=1$. If $\bC_G(V)$ is not ableian, then
by induction there exists $\psi \in \Irr(\bC_G(V))$ such that
$\psi(1)\geq |H \cap \bC_G(V)|$. If $\bC_G(V)$ is ableian, then
clearly all $\psi \in \Irr(\bC_G(V))$ has degree $1$ and $|H \cap
\bC_G(V)|=1$. Thus in all cases, we have $\psi(1)\geq |H \cap
\bC_G(V)|$.





Let $\chi_i \in \Irr(V_i)$ such that $\chi_i(1) = b(V_i)$ and set
$\chi = \chi_1 \times \cdots \times \chi_k$. Clearly $\chi \in
\Irr(V)$ and $\chi(1) = \chi_1^k(1)$. Note that $G/(V \times
\bC_G(V)) \leq \Out(V)$, $\Out(V) \cong \Out(V_1) \wr S_k$, and $G
/\bC_G(V) \leq \Aut(V_1) \wr S_k$. By Lemma~\ref{lem4}, we have
\[|H\bC_G(V)/\bC_G(V)| \leq (|\Aut(V_1)|_{abelian})^k.\]
Suppose that $V_1 \not \cong \Al_5$. By Theorem ~\ref{abeliancase},
we have
$$|\Aut(V_1)|_{abelian})^k \leq \chi_1(1)^{k} = \chi(1).$$
Combining the last two inequalities with $\psi(1)\geq |H \cap
\bC_G(V)|$ and note that $\chi\times \psi$ is an irreducible
character of $V\times
\bC_G(V)$, we obtain \begin{align*}b(G)&\geq b(V\times \bC_G(V))\\
&\geq \psi(1)\chi(1)\\
&\geq |H\cap \bC_G(V)||H\bC_G(V)/\bC_G(V)|\\
&=|H|,\end{align*} and we are done.

It remains to consider the case $V_i \cong \Al_5$. Then there exists
$\mu_i, \nu_i \in \Irr(V_i)$ such that $\mu_i(1)=3$, $\nu_i(1)=3$
and $\bI_{N_G(V_i)}(\mu_i) = \bI_{N_G (V_i)}(\nu_i)= \bC_G(V_i)
\times V_i$. Using Lemma ~\ref{lem3} and the fact that $H$ is
abelian, we see that there exists $\chi = \prod_i\chi_i \in \Irr(V)$
such that $\chi_i \in \{\mu_i, \nu_i\}$ and $\bI_G(\chi) = V\times
\bC_G(V)$. Now $\chi \times \psi \in \Irr(V\times \bC_G(V))$ induces
an irreducible character of $G$ where $\bI_G(\chi \times \psi) =
V\times \bC_G(V)$, and it follows that \begin{align*}|H| &\leq |G :
(V \times \bC_G(V))| |H \cap V| |H \cap \bC_G(V)|\\
& \leq |G : (V \times \bC_G(V))| \chi(1) \psi(1)\\
&\leq b(G).\end{align*} The proof is complete.
\end{proof}

\begin{remark} In ~\cite{Vdovin2,Vdovin1}, Vdovin studied the size and in
some cases structure of a maximal abelian/nilpotent subgroup of a
finite simple group. However his bounds are not sufficient for the
purpose of this paper where we need to bound the size of an
abelian/nilpotent subgroup of an almost simple group. We believe the
results on the size of abelian and nilpotent subgroups of almost
simple groups will be useful in other applications.
\end{remark}

\section{Acknowledgement} \label{sec:Acknowledgement}
This work was initiated when the first author visited the Department
of Mathematics at Texas State University during Fall 2017. He would
like to thank the department for its hospitality. This work was
partially supported by a grant from the Simons Foundation (No
499532, to YY).


\end{document}